\documentclass[12pt,tbtags,leqno]{amsart}

\usepackage{amssymb}
\usepackage{amsthm}
\usepackage{amsmath}
\usepackage{verbatim}
\usepackage{cite}

\numberwithin{equation}{section}

\newcommand{\HH}{\mathcal{H}}

\newcommand{\HB}{\mathcal{B}}

\newcommand{\Hprod}{\mathcal H \odot \mathcal H}

\newcommand{\dB}{\partial \mathbb B_d}

\newcommand{\C}{\mathbb{C}}
\newcommand{\N}{\mathbb{N}}
\newcommand{\R}{\mathbb{R}}

\newcommand{\Bd}{\mathbb{B}_d}

\newcommand{\la}{\langle}
\newcommand{\ra}{\rangle}
\newcommand{\Hol}{\text{Hol}}

\theoremstyle{plain}
\newtheorem{theorem}{Theorem}[section]
\newtheorem{lemma}[theorem]{Lemma}

\newtheorem{prop}[theorem]{Proposition}

\theoremstyle{definition}

\begin{document}

\date{\today}

\bibliographystyle{plain}

\title[Multipliers]{A remark on the multipliers on spaces of Weak Products of functions}
\author{Stefan Richter}
\author{Brett D. Wick}
\address{Stefan Richter, Department of Mathematics \\
       University of Tennessee \\
       Knoxville, TN  37996}
\address{Brett D. Wick, Department of Mathematics \\
       Washington University -- St. Louis \\
       St. Louis, MO 63130}

%\email{richter@math.utk.edu} \subjclass[2000]{Primary 31C25, 47B35; Secondary 30C85}
 \thanks{Work of BDW was supported by the National Science Foundation, grants DMS-1560955 and DMS-0955432.}
\maketitle
\begin{abstract} If $\mathcal{H}$ denotes a Hilbert space of analytic functions on a region $\Omega \subseteq \mathbb{C}^d$, then the weak product is defined by
$$\mathcal{H}\odot\mathcal{H}=\left\{h=\sum_{n=1}^\infty f_n g_n : \sum_{n=1}^\infty \|f_n\|_{\mathcal{H}}\|g_n\|_{\mathcal{H}} <\infty\right\}.$$  We prove that if $\mathcal{H}$ is a first order holomorphic Besov Hilbert space on the unit ball of $\mathbb{C}^d$, then the multiplier algebras of $\mathcal{H}$ and of $\mathcal{H}\odot\mathcal{H}$ coincide.
\end{abstract}

\section{Introduction}

Let $d$ be a positive integer and let $R= \sum_{i=1}^dz_i\frac{\partial}{\partial z_i}$ denote the radial derivative operator. For $s\in \R$ the holomorphic Besov space $B_s$ is defined to be the space of holomorphic functions $f$ on the unit ball $\Bd$ of $\C^d$ such that for some nonnegative integer $k>s$
$$\|f\|^2_{k,s}=\int_{\Bd}|(I+R)^k f(z)|^2 (1-|z|^2)^{2(k-s)-1} dV(z) <\infty.$$ Here $dV$ denotes Lebesgue measure on $\Bd$.
It is well-known that for any $f \in \Hol(\Bd)$ and any $s\in \R$ the quantity $\|f\|_{k,s}$ is finite for some nonnegative integer $k>s$ if and only if it is finite for all nonnegative integers $k>s$, and that for each $k>s$ $\|\cdot\|_{k,s}$ defines a norm on $B_s$, and that all these norms are equivalent to one another, see \cite{BeatrousBurbea}. For $s<0$ one can take  $k=0$ and these spaces are weighted Bergman spaces. In particular, $B_{-1/2}=L^2_a(\Bd)$ is the unweighted Bergman space. For $s=0$ one obtains the Hardy space of $\Bd$ and one has that for each $k\ge 1$
$\|f\|^2_{k,0}$
is equivalent to $\int_{\dB}|f|^2d\sigma$, where $\sigma$ is the rotationally invariant probability measure on $\dB$. We also note that for $s= {(d-1)}/2$ we have $B_s=H^2_d$, the Drury-Arveson space. If $d=1$ and $s=1/2$, then $B_s=D$, the classical Dirichlet space of the unit disc.

Let $\HH \subseteq \text{Hol}(\Bd)$ be a
reproducing kernel Hilbert space such that $1\in \HH$. The weak product of $\HH$ is
denoted by $\HH \odot \HH$ and it is defined to be the collection of
all functions $h \in \text{Hol}(\Bd)$ such that there are
sequences $ \{f_i\}_{i\ge 1}, \{g_i\}_{i\ge 1} \subseteq \HH $ with
$\sum_{i= 1}^\infty \|f_i\|_{\HH}\|g_i\|_{\HH}< \infty$ and for all $z \in \Bd$, $h(z) =\sum_{i=
1}^\infty f_i(z)g_i(z)$.

We define a norm on $\HH \odot \HH$ by
$$\|h\|_* =\inf\left\{\sum_{i=1}^\infty \|f_i\|_{\HH}\|g_i\|_{\HH}: h(z)= \sum_{i=1}^\infty
f_i(z)g_i(z) \text{ for all } z \in \Bd\right\}.$$
In what appears below we will frequently take $\HH=B_s$, and will use the same notation for this weak product.

Weak products have their origin in the work of Coifman, Rochberg, and Weiss \cite{CRW}. In the frame work of the Hilbert space $\HH$ one may consider  the weak product  to be an analogue of the Hardy $H^1$-space. For example, one has $H^2(\dB)\odot H^2(\dB)=H^1(\dB)$ and $L^2_a(\Bd)\odot L^2_a(\Bd)=L^1_a(\Bd)$, see \cite{CRW}.
For the Dirichlet space $D$  the weak product $D\odot D$ has recently been considered in \cite{ARSW}, \cite{CasOrt}, \cite{RiSu}, \cite{Luo}, and \cite{LuoRi}. The space $H^2_d\odot H^2_d$ was used in \cite{RiSunkes}. For further motivation and general background on weak products we refer the reader to \cite{ARSW} and \cite{RiSu}.

Let $\HB$ be a Banach space of analytic functions on $\Bd$ such that point evaluations are continuous and such that $1\in \HB$. We use  $M(\HB)$ to denote the multiplier algebra of $\HB$, $$M(\HB)=\left\{\varphi: \varphi f \in \HB \text{ for all }f\in \HB\right\}.$$ The multiplier norm $\|\varphi\|_M$ is defined to be the norm of the associated multiplication operator $M_\varphi: \HB\to \HB$.
It is easy to check and is well-known that  $M(\HB) \subseteq H^\infty(\Bd)$, and that for $s\le 0$ we have $M(B_s)=H^\infty(\Bd)$. For $s>d/2$ the space $B_s$ is an algebra \cite{BeatrousBurbea}, hence $B_s= M(B_s)$, but for $0<s \le d/2$ one has $M(B_s) \subsetneq B_s\cap H^\infty(\dB).$ For those cases $M(B_s)$ has been described by a certain Carleson measure condition, see \cite{OrtFab,CasFabOrt}.

It is easy to see that $M(\HH)\subseteq M(\Hprod)\subseteq H^\infty$ (see Proposition \ref{MultIncl}).
Thus, if $s\le 0$, then  $M(B_s)=M(B_s\odot B_s)=H^\infty$. Furthermore, if $s>d/2$, then $B_s=B_s\odot B_s= M(B_s)$ since $B_s$ is an algebra. This raises the question whether $M(B_s)$ and $M(B_s\odot B_s)$ always agree. We prove the following:
\begin{theorem}\label{MainTh} Let $s\in \R$ and $d\in \N$. If $s\le 1$ or $d\le 2$, then $M(B_s)=M(B_s\odot B_s)$.
\end{theorem}
Note that when $d\le 2$, then $B_s$ is an algebra for all $s>1$. Thus for each  $d\in \N$ the nontrivial range of the Theorem is $0 <s\le 1$. If $d=1$ then the theorem applies to the classical Dirichlet space of the unit disc and for $d\le 3$ it applies to the Drury-Arveson space.

\section{Preliminaries}
For $z=(z_1, ..., z_d)\in \C^d$ and $t \in \R$ we write $e^{it}z=(e^{it}z_1,...,e^{it}z_d)$ and we write $\la z,w\ra$ for the inner product in $\C^d$. Furthermore, if $h$ is a function on $\Bd$, then we define $T_t f$ by $(T_tf)(z)=f(e^{it}z)$. We say that a space $\HH\subseteq \Hol(\Bd)$ is radially symmetric, if each $T_t$ acts isometrically on $\HH$ and if for all $t_0 \in \R$, $T_t\to T_{t_0}$ in the strong operator topology as $t\to t_0$, i.e. if $\|T_tf\|_{\HH}=\|f\|_{\HH}$ and $\|T_tf-T_{t_0}f\|_{\HH}\to 0$ for all $f\in \HH$. For example, for each $s\in \R$ the holomorphic Besov space $B_s$ is radially symmetric when equipped with any of the norms $\|\cdot\|_{k,s}$, $k>s$.

It is elementary to verify the following lemma.
\begin{lemma} If $\HH \subseteq\Hol(\Bd)$ is radially symmetric, then so is $\Hprod$.
\end{lemma}

Note that if $h$ and $\varphi$ are functions on $\Bd$, then for every $t\in \R$ we have $(T_t\varphi)h= T_t(\varphi T_{-t}h)$, hence if a space is radially symmetric, then $T_t$ acts isometrically on the multiplier algebra. For $0<r<1$ we write $f_r(z)=f(rz)$.

\begin{lemma} If $\HH \subseteq \Hol(\Bd)$ is radially symmetric, and if $\varphi\in M(\Hprod)$, then for all $0<r<1$ we have $\|\varphi_r\|_{M(\Hprod)}\le \|\varphi\|_{M(\Hprod)}$.
\end{lemma}
\begin{proof}  Let $\varphi\in M(\Hprod)$ and $h\in \Hprod$, then for $0<r<1$ we have $$\varphi_rh= \int_{-\pi}^{\pi} \frac{1-r^2}{|1-re^{it}|^2} (T_t\varphi)h \frac{dt}{2\pi}.$$ This implies $$\|\varphi_r h\|_*\le \int_{-\pi}^{\pi} \frac{1-r^2}{|1-re^{it}|^2} \|(T_t\varphi)h\|_* \frac{dt}{2\pi}\le \|\varphi\|_{M(\Hprod)} \|h\|_*.$$ Thus, $\|\varphi_r\|_{M(\Hprod)}\le \|\varphi\|_{M(\Hprod)}$.
\end{proof}

\section{Multipliers}

 The following Proposition is elementary.
\begin{prop} \label{MultIncl} We have $M(\HH)\subseteq M(\Hprod)\subseteq H^\infty$ and if
 $\varphi \in M(\HH)$,  $\|\varphi\|_{M(\Hprod)} \le \|\varphi\|_{M(\HH)}$. \end{prop}

As explained in the Introduction, the following will establish Theorem \ref{MainTh}.

\begin{theorem} Let $0<s\le 1$. Then $M(B_s)=M(B_s \odot B_s)$ and there is a $C_s >0$ such that $$\|\varphi\|_{M(B_s \odot B_s)} \le \|\varphi\|_{M(B_s)} \le C_s \|\varphi\|_{M(B_s \odot B_s)}$$ for all $\varphi \in M(B_s)$.
\end{theorem}
Here for each $s$ we have the norm on $B_s$ to be $\|\cdot\|_{k,s}$, where $k$ is the smallest natural number $>s$.
\begin{proof} We first do the case $0<s<1$.  Then  $k=1$, and $\|f\|^2_{B_s} = \int_{\Bd}|(I+R)f(z)|^2 dV_s(z)$, where $dV_s(z)=(1-|z|^2)^{1-2s}dV(z)$. For later reference we note that a short calculation shows that $\int_{\Bd}|Rf|^2 dV_s \le \|f\|^2_{B_s}$.

 We write $\|R\varphi\|_{Ca(B_s)}$ for the Carleson measure norm of $|R\varphi|^2$, i.e. $$\|R\varphi\|^2_{Ca(B_s)}=\inf\left\{C>0: \int_{\Bd}|f|^2 |R\varphi|^2 dV_s \le C\|f\|_{B_s}^2  \text{ for all }f\in B_s\right\}.$$

Since $\|\varphi f\|^2_{B_s}= \int_{\Bd} |\varphi(z) (I+R)f(z) +f(z)R\varphi(z)|^2 dV_s(z)$ it is clear that $\|\varphi\|_{M(B_s)}$ is equivalent to $\|\varphi\|_\infty+\|R\varphi\|_{Ca(B_s)}$. Thus, it suffices to show that there is a $c>0$ such that $\|R\varphi\|_{Ca(B_s)} \le c\|\varphi\|_{M(B_s \odot B_s)}$ for all $\varphi \in M(B_s\odot B_s)$.

First we note that if $b$ is holomorphic in a neighborhood of $\overline{\Bd}$ and
 $h= \sum_{i=1}^\infty f_ig_i \in B_s \odot B_s$, then
\begin{align*} \int_{\Bd}|(Rh)Rb|dV_s &\le \sum_{i=1}^\infty \int_{\Bd} |(Rf_i)g_i Rb|dV_s+ \int_{\Bd} |(Rg_i)f_i Rb|dV_s\\
&\le  \sum_{i=1}^\infty \|f_i\|_{B_s} \left(\int_{\Bd} |g_i Rb|^2dV_s \right)^{1/2} +\|g_i\|_{B_s} \left(\int_{\Bd} |f_i Rb|^2dV_s \right)^{1/2}\\
&\le  2\sum_{i=1}^\infty \|f_i\|_{B_s}\|g_i\|_{B_s} \|Rb\|_{Ca(B_s)}. \end{align*}
Hence
$$ \int_{\Bd}|(Rh)Rb|dV_s \le 2\|h\|_* \|Rb\|_{Ca(B_s)},$$  where we have continued to write $\|\cdot\|_*$ for $\|\cdot\|_{B_s\odot B_s}$.

Let $\varphi \in M(B_s \odot B_s)$ and let $0<r<1$. Then for all $f \in B_s$ we have $f^2, \varphi_rf^2 \in B_s \odot B_s$, hence
\begin{align*} \int_{\Bd}|f|^2 |R\varphi_r|^2dV_s &= \int_{\Bd}|R(\varphi_rf^2)-\varphi_rR(f^2)|\ |R\varphi_r|dV_s \\
&\le 2 (\|\varphi_rf^2\|_* +\|\varphi\|_\infty \|f^2\|_*)\|R\varphi_r\|_{Ca(B_s)}\\
&\le 2 (\|\varphi\|_{M(B_s \odot B_s)} \|f^2\|_* +\|\varphi\|_\infty \|f^2\|_*)\|R\varphi_r\|_{Ca(B_s)}\\
&\le 4 \|\varphi\|_{M(B_s \odot B_s)} \|f\|^2_{B_s} \|R\varphi_r\|_{Ca(B_s)}.
\end{align*}
Next we take the sup of the left hand side of this expression over all $f$ with $\|f\|_{B_s}=1$ and we obtain $\|R\varphi_r\|^2_{Ca(B_s)} \le 4 \|\varphi\|_{M(B_s \odot B_s)}  \|R\varphi_r\|_{Ca(B_s)}$ which implies that $\|R\varphi_r\|_{Ca(B_s)} \le 4 \|\varphi\|_{M(B_s \odot B_s)}$ holds for all $0<r<1$. Thus, for $0<s< 1$ the result follows from Fatou's lemma as $r\to 1$.

If $s=1$, then $\|f\|^2_{2,1} \sim \int_{\dB}|(I+R)f(z)|^2 d\sigma(z)$ and the argument proceeds as above.
\end{proof}

\include{biblio}
\bibliography{MultWeakProd}

\end{document}